\documentclass{article}

\usepackage{graphicx} 
\usepackage{float}
\usepackage{amssymb,amsbsy,amsthm,amsmath,epsfig,babel}
\usepackage{authblk}
\usepackage{enumitem}   

\newtheorem{propo}{Proposition}
\newtheorem{corol}{Corollary}
\newtheorem{defin}{Definition}
\newtheorem{theor}{Theorem}
\newtheorem{conje}{Conjecture}
\newtheorem{remar}{Remark}

\title{A Complete Congruence System for the Erdos-Straus Conjecture}
\author[$\dagger$]{Miguel Angel Lopez}
\affil[$\dagger$]{Complutense University of Madrid,
Plaza de Ciencias 3,
Madrid 28040,
Spain}

\date{}

\begin{document}

\maketitle

\begin{center}email: migulo23@ucm.es, dagon.magnus@gmail.com\end{center}

\begin{abstract}
    Abstract: In this paper we attack the Erdos-Straus conjecture by means of the structure of its solutions, extending and improving the results of a previous paper. Using previous results and supported by the works of Elsholtz and Tao and Monks and Velingker we define a system of congruences for which there are always solutions to the Erdos-Straus conjecture and which we conjecture to include all prime numbers. For this purpose, and always taking into account a result due to Mordell that limits the congruences admitting polynomial identities to those that are not quadratic residues, we will adopt a transversal approach and classify the solutions by their form and not by those congruences that produce them. Thus we define two new types of solutions, which we call Type A and B, and relate them to the already known Type II solutions and study their properties. Finally we conjecture that every prime number has at least one solution of Type A or B and we associate a congruence and a general polynomial to each Type of solution.
\end{abstract}

\textbf{Keywords:} Erdos-Straus, diophantine equation, unit fraction, egyptian fraction, congruences, integral solution, prime numbers, quadratic residues

\textbf{Mathematics Subject Classification:} 11D72, 11A07, 11A41

\section{Introduction}

In 1948, Paul Erdös and Ernst G. Straus formulated a conjecture that states the following: the equation $$\frac{4}{n}=\frac{1}{x}+\frac{1}{y}+\frac{1}{z}$$ has at least one solution where $x,$ $y,$ and $z$ are positive integers. There are many modular identities that solve this equation, for example for the case $n\equiv 2 \pmod3$ we can use the following expression: $$\frac{4}{n}=\frac{1}{n}+\frac{1}{(n+1)/3}+\frac{1}{n(n+1)/3}$$ L. J. Mordell studied this equation among many others in \cite{mordell_diophantine_1970} and derived identities for the cases $p\equiv 3 \pmod4$, $2$ or $3\pmod5$, 5 or $6\pmod7$ and $5\pmod8$. These combined identities solve all cases except those where $p$ is congruent to 1, 121, 169, 289, 361 or $529\pmod{840}$. Mordell himself wondered whether it might not be possible to find a sufficient number of identities such that all possible cases would be completely covered. This possibility was limited, however, by his discovery that if an identity exists for a set of values $p\equiv r \pmod{q}$ then $r$ cannot be a quadratic residue module $q$. For example, no such identity can exist for values of $p$ congruent to $1 \pmod{q}$  since 1 is always a quadratic residue module $q$ for any natural value of $q$. This implies, in fact, that it is not possible to find a value $q$ such that identities can be found for all elements of $\mathbb{Z}_q$.
\\ \\
The conjecture has been verified up to value $10^7$ by Yamamoto \cite{yamamoto_diophantine_1965} and $10^{14}$ by Swett \cite{salez_erdos-straus_2014}. Webb and others have shown that the natural density of possible counterexamples to the conjecture is zero, for as $N$ tends to infinity, the number of values in the interval $[1,N]$ that could be counterexamples tends to zero.
\\ \\
In \cite{lopez_structure_2022} we showed that, if $p=4k+1$, there exists a solution for $p$ such that $$\frac{4}{p}=\frac{1}{du}+\frac{1}{dv}+\frac{1}{duv}$$ if and only if there exists $t\geq0$  and a divisor $w$ of $k+1+t$ such that $w\equiv -1 \pmod{3+4t}$. In particular, if there exists a divisor $w$ of $k+1$ congruent to 2 module 3, $p$ has the previously mentioned solution. This result is powerful even in the particular case where $t=0$ and in fact only by employing this particular case all cases are ruled out except those where $p\equiv 1 \pmod{24}$ and also works in many other cases within this congruence. That the problem cannot be solved completely with this result is logical not only because, as Mordell rightly points out, 1 is quadratic residue modulo 24, but also because this congruence contains the perfect square of every prime number greater than 3, with repercussions that we will explain later.

\section{Main body}

As in the previous article, we will establish a similar notation. We will consider the values of each possible solution in an increasing order, i.e. $x \leq y \leq z$. We will write a solution for each value $a \in \mathbb{N}$  as  what will allow us to list them in a more condensed form. We will use the following definition:

\begin{defin}
We say that a number $a \in \mathbb{N}$ is Egyptian of order 3 if there exists a triplet of the form $(x,y,z)$ such that the Erdos-Straus conjecture is fulfilled for the fraction $\frac{4}{a}$.    
\end{defin}

We will also use the following definition, used by Bradford in \cite{bradford_note_2020} and by Elsholtz and Tao in \cite{elsholtz_counting_2015}:
\begin{defin}
We will say that a solution $(x,y,z)$ is of Type I if $gcd(a,x)=gcd(a,y)=1$ and $gcd(a,z)=a$ while a solution will be of Type II if $gcd(a,y)=gcd(a,z)=a$ and $gcd(a,x)=1$.    
\end{defin}

It is well known that prime numbers can only have these two types of solutions. We will also point out that, contrary to the previous article, in some cases we will consider the value a to be studied as a composite, and when it is prime we will use the letter $p$. We will begin by recalling the result of the previous article that motivates the new approach to the conjecture, and it is the following.

\begin{theor}
\label{T1}
Let be $p \in \mathbb{N}$ prime. Suppose that $p=4k+1$, there exists a solution for $p$ with the form $(du,dv,duv)$ with $d,u,v \in \mathbb{N}$ if and only if there exists $t \geq 0$ and a divisor $w$ of $k+1+t$ such that $w \equiv -1 \pmod{3+4t}$. In particular, if there exists a divisor $w$ of $k+1$ congruent to 2 module 3, $p$ is Egyptian of order 3 with a solution of the form $(du,dv,duv)$.
\end{theor}

We will not prove this result, for anyone interested in seeing its proof may consult \cite{lopez_structure_2022}. The proof is also constructive in an implicit way; following its steps, it follows that $$d=\frac{k+1+t}{w},n=\frac{w+1}{3+4t},u=\frac{1+np}{4dn-1},v=np$$ All these numbers are natural in a trivial way, without more than applying the hypotheses, except in the case of $u$, for which it is simple to verify it by substituting the values to arrive at $1+np \equiv 0 \pmod{4dn-1}$. It can be verified easily also that $\frac{4}{p}=\frac{1}{du}+\frac{1}{dv}+\frac{1}{duv}$ because it is a mere algebraic exercise.
\\ \\
It is also evident, without more than performing a modular arithmetic calculation that, for the case $t=0$, if $k+1$ has a divisor $w$ congruent to 2 modulo 3, then either $w$ itself is a prime number or there exists another divisor of $k+1$ congruent to 2 modulo 3 that is, so we can restrict our search for valid $w$ divisors strictly to prime candidates.
\\ \\
As a result of the proof, in addition, the paper proved that for any value of $t$ it is satisfied that $gcd(u,v)=1$. We can also come to a quick conclusion as to whether it is one of the basic types of solution.
\\ \\
\begin{theor}
\label{T2}
Let be $p \in \mathbb{N}$ prime of the form $p=4k+1$, if $p$ has a solution $(du,dv,duv)$, satisfies that this solution is of Type II.    
\end{theor}

\begin{proof}
    Suppose without loss of generality that $u \leq v$. Monks and Velingker prove in \cite{monks_erdos-straus_2007} that necessarily, since $du \leq dv$, $p$ does not divide $du$. Since we know that $$\frac{4}{p}=\frac{1}{du}+\frac{1}{dv}+\frac{1}{duv}$$ and therefore, $$4duv=p(1+u+v)$$ we have then that $p \mid 4v$, and since $p$ is always odd, being of the form $p=4k+1$, then $p \mid v$. This automatically implies that the solution $(du,dv,duv)$ is of Type II.
\end{proof}

This reasoning can be replicated without variations for the case $p=4k+3$ and it is trivial to observe that the condition is also fulfilled for the only solution of the case $p=2$. The first limitation regarding this solution is that it can never be assumed with these same properties for any value that lacks Type II solutions, and such values, as Elsholtz and Tao already commented, exist.
\\ \\
At this point, we will name these types of solutions.

\begin{defin}
We say that a value $a \in \mathbb{N}$ has a solution of Type A if there exist $d,u,v \in \mathbb{N}$ such that $a$ is Egyptian of order 3 with a solution $(du,dv,duv)$. This solution is of Type II if $a$ is prime.
\end{defin}

In the previous article it was shown that any prime number $p$ such that $p\not\equiv 1 \pmod{24}$ always has a solution of Type A. For this, it was sufficient to use (\ref{T1}) and the particular case where $t=0$. It is possible to keep digging in this direction until one finds huge lists of congruences with Type A solutions associated, since it has not even been used in particular cases other than the value of $t=0$. There are many papers that develop large numbers of identities, but that approach has many limitations. It is interesting to note that the congruence $p \equiv 1 \pmod{24}$ includes the values 25,49,121,169,289\dots All of them are perfect squares, which is no coincidence.

\begin{theor}
\label{T3}
Let $n$ be odd such that $n\geq5$, $3 \nmid n$.Then $n^2 \equiv 1 \pmod{24}$. In particular, the square of every prime number greater than 3 belongs to this congruence.    
\end{theor}

\begin{proof}
    We consider the value $n^2-1=(n-1)(n+1)$. Since $n$ is odd, we know that $n\equiv 1,3 \pmod4$. If $n\equiv 1 \pmod4$ then $n+1\equiv 2 \pmod4$ and $n-1\equiv 0 \pmod4$, which instantly implies that $n^2\equiv 1 \pmod8$. If $n\equiv 3 \pmod4$ we reach the same conclusion by analogous reasoning. On the other hand, since $3\nmid n$ we have that $n\equiv 1,2 \pmod3$. If $n\equiv 1 \pmod3$ then $n-1\equiv 0 \pmod3$ and therefore $n^2-1\equiv 0 \pmod3$. Likewise, if $n\equiv 2 \pmod3$ then $n+1\equiv 0 \pmod3$ and therefore $n^2-1\equiv 0 \pmod3$. In both cases $n^2\equiv 1 \pmod3$. Thanks to the Chinese Remainder Theorem we finally have that $n^2\equiv 1 \pmod{24}$.
\end{proof}

This result we have just proved extends an earlier classical one by Conway and Guy published in \cite{conway_book_1996} that said that the square of every odd number is congruent to 1 modulo 8. The fact that this happens imposes limits on what can be achieved with Type I and II solutions, as Schinzel and Yamamoto point out, since they proved that the square of every natural number lacks both Type I and Type II solutions. For a modern proof, see \cite{elsholtz_counting_2015}.
\\ \\
Solutions of Type A are transversal to the identity of prime and composite; both types of numbers can possess a solution with such a structure. If a prime number $p$ possesses a solution of the form $(du,dv,duv)$ then every multiple $mp$ with $m>1$ possesses a solution of the form $(mdu,mdv,mduv)$ which still possesses the structure $(Du,Dv,Duv)$ with $D=md$. The main difference lies in the fact that, while the first solution is necessarily of Type II, the second does not belong to that type since $gcd(mp,Du)\geq m$.
\\ \\
It is also immediate to verify that, if we try to expand (\ref{T1}) for composite values $a$ of the form $a=4k+1$, then the implication from right to left is immediately fulfilled (it is enough to take $d=\frac{k+1+t}{w}$, $n=\frac{w+1}{3+4t}$, $u=\frac{1+np}{4dn-1}$, $v=np$, as we said before), while the opposite implication is false, for which it is enough to look for example at the value $a=25$, that has solutions of Type A but for which there is no $t\geq 0$ such that $6+1+t$ has a divisor $w$ such that $w\equiv -1 \pmod{3+4t}$. These solutions can never be of Type II for this value, being 25 a perfect square.
\\ \\
If a thorough analysis of prime numbers is performed, it can be seen that, among the first 9000 natural numbers, the only prime numbers lacking Type A solutions are 193 and 2521. In many articles these values suddenly appear as numbers that offer various types of problems or resist classification, as for example \cite{ionascu_erdos-straus_2010}. These two numbers, however, have a second common structure of the form $d(uv,up,vp)$. Both numbers possess solutions of this type; 193 has e.g. $$(50,1930,4825)=5(10,2\cdot193,5\cdot193)$$ and for 2521 we have the tern $$(638,55462,804199)=11(58,2\cdot2521,29\cdot2521)$$ This will be our second structure to analyze.

\begin{defin}
We say that a value $a \in \mathbb{N}$ has a solution of Type B if there exist $d,u,v \in \mathbb{N}$ such that $a$ is Egyptian of order 3 with a solution $(duv,dua,dva)$.    
\end{defin}

First, let us characterize in a simple way what a prime number must fulfill to possess a solution of this type, and we will find an interesting parallelism with the previous case.

\begin{theor}
\label{T4}
Let be $p \in \mathbb{N}$ prime, let $d \in \mathbb{N}$. There exists a Type B solution for $p$ if and only if it exists $n \in \mathbb{N}$  such that $p\equiv -n \pmod{4dn-1}$.  In addition this solution is of Type II. 
\end{theor}

\begin{proof}
    First we will prove the implication from left to right. We have a solution that satisfies the equation $$\frac{4}{p}=\frac{1}{duv}+\frac{1}{dup}+\frac{1}{dvp}$$ and which we can rewrite as $$4duv=p+u+v$$ This implies that $$u=\frac{p+v}{4dv-1}$$ and this number will only be natural if $4dv-1\mid p+v$ or, what is the same thing, there exists a certain $n \in \mathbb{N}$ such that $$p+n \equiv 0 \pmod{4dn-1}$$

    To prove the converse implication, let $n \in \mathbb{N}$ be such that the above congruence is satisfied for a certain value $d\in\mathbb{N}$.  We define $u=\frac{p+n}{4dn-1}$, $v=n$. Both numbers are natural and it’s straigthforward to prove that $$\frac{4}{p}=\frac{1}{duv}+\frac{1}{dup}+\frac{1}{dvp}$$ Moreover, it always happens that $gcd(p,duv)=1$, otherwise $p^2$ would be divisor of at least one of the three coordinates of the solution and such a thing is impossible, as demonstrated by Monks and Velingker in \cite{monks_erdos-straus_2007}. That classifies these solutions again in the Type II category, which was the last thing we wanted to prove.
\end{proof}

An immediate and very interesting conclusion is that this characterization does not use the fact that $p$ is a prime number, but it is necessary to consider it to demonstrate that this solution is of Type II, since there are composite numbers that have solutions of Type B but are not of Type II, for example $n=6$. On the other hand there are also composite numbers that do not have solutions of this type, for example $n=15$.  As a joint conclusion, we have that all the solutions of Type A and B are in particular of Type II when they are referred to a prime number.
\\ \\
Again, Type B solutions are transversal to the identity of prime and composite; both types of numbers can possess a solution with such a structure. However, unlike the previous case, if a prime number $p$ has a solution of the form $d(uv,up,vp)$ this does not imply that every multiple $mp$  with $m>1$  has a solution of the form  $dm(uv,up,vp)$.
\\ \\
As we have already mentioned above, perfect squares can never have solutions of Type I or II. This implies, first of all, that if a perfect square $a$ has a Type A solution, this solution cannot be of Type II and, therefore, must be inherited from one of its divisors.
\\ \\
On the other hand, a stronger conclusion is the following: 

\begin{theor}
\label{T5}
A perfect square $s$ can never have a Type B solution,, whether or not it is of Type II.    
\end{theor}

\begin{proof}
    If it does, then using (\ref{T4}) we have that there exist $d,n \in \mathbb{N}$ such that $$s\equiv -n \pmod{4dn-1}$$ and therefore, following the steps of the theorem, we construct the identity $$\frac{4}{s}=\frac{1}{dn\frac{s+n}{4dn-1}}+\frac{1}{sd\frac{s+n}{4dn-1}}+\frac{1}{sdn}$$ This identity is of the form $$\frac{4}{s}=\frac{1}{P_{d,n}(s)}+\frac{1}{Q_{d,n}(s)}+\frac{1}{R_{d,n}(s)}$$ where all these functions are polynomial in the variable $s$, and Mordell proved that, then, $-n$ cannot be a quadratic residue module $4dn-1$, but it is, since $s$ is a perfect square, and this is a contradiction. Therefore no perfect square ever possesses a Type B solution.
\end{proof}

Another immediate conclusion is the following:

\begin{corol}
Let $d,n \in \mathbb{N}$, then $\left( \frac{-n}{4dn-1} \right)=-1$ with $\left(\frac{a}{b}\right)$ equal to the Jacobi symbol. In particular if $4dn-1$ is prime then $\left(\frac{n}{4dn-1}\right)=1$ and $n$ is always quadratic residue module $4dn-1$ (but may or may not be if $4dn-1$ is a composite number).    
\end{corol}

We can also characterize the Type B solutions for the particular case of prime numbers $p=4k+1$.

\begin{theor}
\label{T6}
Let be $p \in \mathbb{N}$ prime. Suppose that $p=4k+1$, there exists a solution for $p$ with the form $(duv,dup,dvp)$ with $d,u,v \in \mathbb{N}$ if and only if there exists $t\geq0$ and two positive divisors $a,b$ of $k+1+t$ such that $a+b=3+4t$.
\end{theor}

\begin{proof}
    We suppose first that $p$ has a Type B solution, then exists $d,n \in \mathbb{N}$ such that $$p\equiv -n \pmod{4dn-1}$$ then $$4k\equiv -n-1 \pmod{4dn-1}$$ and we can deduce that $$k\equiv -dn^2-dn \pmod{4dn-1}$$ This implies that exists $c \in \mathbb{Z}$ with $k+dn^2+dn=c(4dn-1)$. We can also suppose that $c$ is natural because $4dn-1$ and $k+dn^2+dn\in \mathbb{N}$. If we reorder the expression, we have that $$k+c=dn(4c-1-n)$$ We define $c=t+1$, $t\geq0$, so $k+1+t=dn(3+4t-n)$. We define $a=n$, $b=3+4t-n$.We have trivially then that $a,b \in \mathbb{N}$, $ab\mid k+1+t$, $a+b=3+4t$.
    \\ \\
    Conversely, lets suppose that $\exists a,b \in \mathbb{N}$, $t\geq0$ such that $ab\mid k+1+t$, $a+b=3+4t$. We can reverse all the process defining $a=n$, $b=3+4t-a$. We define $$z=n(3+4t-n)=ab$$ then $$k+1+t=dz, d\in \mathbb{N}$$ and therefore $$k+1+t=3dn+4dnt-dn^2$$ We define $t=c-1$, $c\geq1$ because $t\geq0$. Then $k+c=4dnc-dn-dn^2$ and rearranging we arrive to $k+dn+dn^2=c(4dn-1)$, $c\geq1$. It is straightforward then that $$k\equiv -dn-dn^2 \pmod{4dn-1}$$ and this implies automatically that there are $d,n \in \mathbb{N}$ such that $p\equiv -n \pmod{4dn-1}$ and therefore $p$ has a Type B solution.
\end{proof}

\begin{corol}
The theorem also holds if $p=4k+1$ but it’s not prime (we never use the primality of $p$ in the demonstration).
\end{corol}

In our previous paper, to obtain the characterization of (\ref{T1}), we employed the following prior result, whose proof, similar to (\ref{T4}) but more complex in some steps, can be read in \cite{lopez_structure_2022}.

\begin{theor}
\label{T7}
Let be $p \in \mathbb{N}$ prime. There exists a Type A solution for $p$ if and only if it exists $d,n \in \mathbb{N}$ such that $p\equiv -4d \pmod{4dn-1}$. In addition the values $u$ and $v$ are coprime and the solution is of Type II.
\end{theor}

Experimental evidence shows that every prime number always has at least one solution with one of these two types of structure. This leads us to formulate what is the main conjecture of this article:

\begin{conje}
Let be $p \in \mathbb{N}$ prime, then exists $d,n \in \mathbb{N}$ such that $p\equiv -4d \pmod{4dn-1}$ or $p\equiv -n \pmod{4dn-1}$. If this result is true, then this congruence system covers all primes and the Erdos-Straus conjecture is true.    
\end{conje}

It is very interesting to note that there is a modular relationship between the type $-n$ and type $-4d$ values, since both sets of divisors are mutual inverses in $\mathbb{Z}_{4dn-1}$.
\\ \\
The conjecture has been experimentally verified and is true for any prime number less than or equal to 104729, that is, it has been verified for the first 10000 prime numbers. Minor cases like $p=2$  or $p=3$  also satisfy the conjecture. The number of solutions does not have a regular growth pattern: a value as large as 83449 not only lacks Type B solutions, but also has only two Type A solutions that are at the same time of Type II (and they are, since 83449 is a prime number, the only two of this type that it has, since it cannot inherit others from any factor).
\\ \\
Both types of structure are necessary for the conjecture to be true: as previously mentioned, 193 and 2521 do not have Type A solutions but do have Type B solutions, and these are not isolated cases, since, for example, 66529 is another case. On the contrary, 23929 does not have Type B solutions but it does have Type A solutions. Both types of solutions, therefore, complement each other.
\\ \\
It may be asked whether it is possible to choose $d,n \in \mathbb{N}$ such that $4dn-1$ is always a prime number. While such a thing is possible for a huge number of cases of $p$, it is not possible to do so in general, and a significant example is, again, the case $p=2521$. This number has only one Type B solution which has two $4dn-1$ associated values, which are 87 and 1275, neither of which is a prime number. 2521 does, in fact, satisfy a property that may perhaps explain in part why it is so elusive: all values from 1 to 10 are quadratic residues in $\mathbb{Z}_{2521}$. In fact, all prime numbers congruent to 1 module 840 fulfill it, precisely one of the congruences that resisted Mordell when studying the conjecture and one of the most resistant to being categorized, perhaps in part because of this special property.
\\ \\
Elsholtz and Tao relate in \cite{elsholtz_counting_2015} the Type II solutions to six-coordinate vectors $(a,b,c,d,e,f)\in \mathbb{C}^6$ satisfying a system of equations. They call the set of all these values $\Sigma_n^\pi$ and define an application $\Pi_n^\pi$ between this set and the algebraic surface $S_n=\{{(x,y,z)\in\mathbb{C}^3:4xyz=nyz+nxz+nxy}\}$ as follows: $$\Pi_n^\pi(a,b,c,d,e,f)=(abd,acdn,bcdn)$$ These solutions are not just Type II, but $abd$ must be required to be coprime with $n$. This condition becomes automatic when $n$ is a prime number, but is not held in general.
\\ \\
It can be immediately observed that Type A solutions, with their values ordered in any possible way, are given when $a=1$ or $b=1$ with the identification $u=b$, $v=cn$ in the case $a=1$ and analogously when $b=1$. When $c=1$ we obtain, precisely, the Type B solutions, performing this time the automatic identification $u=a$, $v=b$. This fact suggests that these two forms of solutions are canonically relevant, being obtained by substituting for the smallest natural values for each of the three main parameters of the Type II solutions. Also, if we allow $d=1$, we will obtain a very relevant result, as we will see later.
\\ \\
The conjecture concerning the existence or non-existence of solutions of Type A and B can be studied under a purely polynomial point of view, associating to each type of solution an algebraic expression, as can be seen in the following theorem.

\begin{theor}
\label{T8}
Let $p$ be prime congruent to 1 modulo 4, we define the following polynomials for $x,y,z,t \in \mathbb{N}$:    

\begin{itemize}
    \item $P(x,y,t)=(4xy-1)(3+4t)-4x^2y$, $x,y,t\geq0$
    \item $Q(x,y,t)=(4xy-1)(3+4t)-4y$, $x,y,t\geq0$
    \item $R(x,y,t,z)=(4xyz-1)(3+4t)-4x^2y$, $x,y,z,t\geq0$
\end{itemize}

\begin{enumerate}[label=(\roman*)]
    \item If $p=P(x,y,t)$ then $p$ has a Type B solution.
    \item If $p=Q(x,y,t)$ then $p$ has a Type A solution.
    \item If $p=R(x,y,t)$ then $p$ has a Type II solution.
\end{enumerate}

All the implications, moreover, also hold conversely: if $p$ possesses a solution of any of the above types, then it necessarily belongs to the image of the corresponding polynomial.
\end{theor}

\begin{proof}
    It is immediate to check, first, that the image of the three polynomials when their variables take natural values are always numbers congruent to 1 modulo 4. Now, we start with (i). We know that $$p\equiv -x \pmod{4xy-1}$$ which is equivalent by (\ref{T4}) to having a Type B solution. Now we consider $p=4k+1$ with a Type B solution, by (\ref{T6}) we know that there exists $a,b$ divisors of $k+1+t$ such that $a+b=3+4t$. We rename $a=x$, $b=3+4t-x$, then exists a certain $y$ such that $$k+1+t=xy(3+4t-x)$$ and then, with elementary algebra, $$p=4k+1=(4xy-1)(3+4t)-4x^2y$$
    To prove (ii), if we define $p=(4xy-1)(3+4t)-4y$, it is immediate to prove that $$p\equiv -4y \pmod{4xy-1}$$ which by (\ref{T7}) implies that if $p$ is prime, then it has a Type A solution. To prove the opposite, we consider that $p=4k+1$ has a solution of Type A, then by (\ref{T1}) there exists a divisor $w$ of $k+1+t$ such that $w$ is congruent to $-1$ modulus $3+4t$. This implies that $w+1=c(3+4t)$  for a certain $c\geq1$, and therefore there exists a certain natural number $e$ such that $k+1+t=ew=e(c(3+4t)-1)$. Rearranging we arrive at $$k+1+t=ce(3+4t)-e$$ which implies that $$p=4k+1=(4ce-1)(3+4t)-4e$$ By renaming $x=c$, $y=e$,we already have it.
    \\ \\
    To prove (iii) we rely on a classic result of Mordell, which can be read for example in \cite{mordell_diophantine_1970}, which says that a value $n$ has a Type II solution if and only if there exist natural values $a,b,c,d$ such that they satisfy $$(4abcd-1)d=an+b$$ By clearing we obtain that $$n=4bcd-\frac{d+b}{a}$$ which will be a natural number if and only if there exists a certain $w$ such that $d+b=aw$. By making the change of variable $d=aw-b$ we get that $$n=4abc(aw-b)-w$$ and this number can only be congruent to 1 modulo 4 if $w$ is congruent to 3 modulo 4 or, what is the same, there exists a certain positive value $s$ for which $w=3+4s$. By making a new change of variable we get that $$n=4abc(a(3+4s)-b)-(3+4s)$$ which, renaming the variables and with some easy manipulations, leads us to the polynomial $R(x,y,z,t)$ we were looking for. The whole procedure is immediately reversible so that the implication works both ways.
\end{proof}

\begin{remar}
It is immediate to check that

\begin{itemize}
    \item $R(x,y,t,1)=P(x,y,t)$
    \item $R(1,y,t,x)=Q(x,y,t)$
\end{itemize}

\end{remar}

\begin{remar}
A p-value has a solution that is both of Type A and B in the particular case in which $x=1$; in that case $P(x,y,t)=Q(x,y,t)$ and, therefore, $0=P(x,y,t)-Q(x,y,t)=4y(x+1)(x-1)$, that can only happen when $x=1$. Below we offer an alternative way of characterizing these solutions that are both of Type A and B.
\end{remar}

We mention now some general statements of Type A and B solutions concerning bounds for $d$ and other simple properties.

\begin{propo}
Let $p$ be a prime number congruent to 1 module 4. If $p$ has a solution of Type A then $d\leq \lfloor\frac{p+3}{8} \rfloor$. This bound is also optimal.
\end{propo}

\begin{proof}
    By (\ref{T1}), we know that if $p$ has a solution of Type A then there exists $t>0$ and $w\equiv -1 \pmod{3+4t}$ such that $w\mid k+1+t$ and, in fact, $d=\frac{k+1+t}{w}$. We can write $w=-1+n(3+4t)$ with $n\geq1$. Then $d=\frac{k+1+t}{3n+4nt-1}$. Therefore, $$d=\frac{k+1+t}{3n+4nt-1}\leq\frac{k+1+t}{4t+2}\leq\frac{k+1}{2}=\frac{p+3}{8}$$ and then we have that $d\leq \lfloor\frac{p+3}{8} \rfloor$.The coordinate is also reached whenever $t=0$, $n=1$, what implies that $w=2$ and $2\mid k+1$, and therefore k is odd. For example, it is reached for $p=13$, for this value $k=3$ and $d=\frac{k+1}{2}=2=\frac{p+3}{8}=\lfloor\frac{p+3}{8} \rfloor$. Therefore the bound is optimal.
\end{proof}

\begin{propo}
Under the conditions of Proposition 1, if $t$ is such that $k+1+t$ has a divisor $w$ that satisfies that $w\equiv -1 \pmod{3+4t}$ then $t\in[{0,\lfloor \frac{k-1}{3} \rfloor }]$.    
\end{propo}

\begin{proof}
    Suppose that $t>\frac{k-1}{3}$, then $3t>k-1$ and therefore $3+4t>k+2+t$. This then implies that $k+1+t\not\equiv -1 \pmod{3+4t}$ and, in particular, any divisor $w$ of $k+1+t$ is less than $3+4t$, so there is no divisor $w$ of $k+1+t$ such that it is satisfied that $w\equiv -1 \pmod{3+4t}$. We have therefore proved the contrapositive of what we were looking for.
\end{proof}

\begin{remar}
There are many examples for which $t$ takes the maximum value of the interval, e.g. $k=4$. In fact, the bound is always optimal whenever $k\equiv 1 \pmod{3}$, because if we consider $k=3u+1$ with $u\geq0$ we have that, if we take $t=u$, $$k+1+t=3u+1+1+u=2+4u$$ $$3+4t=3+4u$$ and in this case we have that $k+1+t\equiv -1 \pmod{3+4t}$ and therefore we can take $w$ as $k+1+t$.    
\end{remar}

The propositions 3 and 4 shows that there are also parallels between Type A and B solutions when calculating their coordinates for $d$.

\begin{propo}
Let $p$ be an odd prime number, if $p$ has a solution of Type B, $(duv,dup,dvp)$ then $d\leq \lfloor\frac{p+3}{8} \rfloor$. This bound is also optimal.    
\end{propo}

\begin{proof}
    We know that if $p$ has a Type B solution, then the equation $4duv=p+u+v$ is satisfied and therefore $d=\frac{p}{4uv}+\frac{1}{4u}+\frac{1}{4v}$. We know that there are multiple values where $u=1$ and $v=2$ or vice versa (by symmetry, they give rise to the same solution), and for those values we have that $d=\frac{p}{8}+\frac{1}{8}+\frac{1}{4}=\frac{p+3}{8}$. For any larger values of $u$ and $v$ we have therefore that $d\leq \lfloor\frac{p+3}{8} \rfloor$.
\\ \\
On the other hand, if $u=v=1$ then, using the construction of (\ref{T4}), we have that $n=v=1$, $u=\frac{p+n}{4dn-1}=\frac{p+1}{4d-1}=1$ and therefore $4d-1=p+1$. But this then implies that $p$ is even, and by hypothesis $p$ is always odd, so the above bound is correct.
\end{proof}

\begin{propo}
Under the conditions of Proposition 3, if there exists $t\geq0$ and two positive divisors $a,b$ of $k+1+t$ such that $a+b=3+4t$ then $t\in[{0,\lfloor \frac{k-1}{3} \rfloor }]$.    
\end{propo}

\begin{proof}
    It’s obvious to prove that, given $n \in \mathbb{N}$, $max\{a+b/ab\mid n\}=n+1$. Therefore there can’t be such divisors for $k+1+t$ when $max\{a+b/ab\mid k+1+t\}<3+4t$, and this implies that $(k+1+t)+1<3+4t$, which implies that $3+4t>k+2+t$, the same values for which we couldn’t find a Type A solution, so again $t$ must belong to the interval $[{0,\lfloor \frac{k-1}{3} \rfloor }]$.
\end{proof}

\begin{propo}
Let $a$ be any natural number, then if $a\equiv -n \pmod{4dn-1}$ we also have that $a\equiv -\left(\frac{a+n}{4dn-1}\right) \pmod{\left(\frac{4ad+1}{4dn-1}\right)}$ and the converse is also true. Both congruences lead us, moreover, to the same Type B solution for a.
\end{propo}

\begin{proof}
    In the first case, using (\ref{T4}), we arrive at a Type B solution of the form $(duv,dua,dva)$ with $u=\frac{a+n}{4dn-1}$, $v=n$. In the second case we have that $4d\left(\frac{a+n}{4dn-1}\right)-1=\frac{4ad+1}{4dn-1}$, so using the same theorem we arrive at another Type B solution of the form $(duv,dua,dva)$, with $v=\frac{a+n}{4dn-1}$ and $u=n$, because $$u=\frac{a+v}{4dv-1}=\frac{a+\frac{a+n}{4dn-1}}{4d\frac{a+n}{4dn-1}-1}=\frac{4adn-a+a+n}{4dn-1}:\frac{4ad+1}{4dn-1}=\frac{4adn+n}{4dn+1}=n$$ Both solutions are the same except for a translation of their coordinates, and therefore both congruences must be true or false simultaneously. The operations also show that they are cyclic, and by means of one congruence the other can be reconstructed and vice versa.
\end{proof}

This last proposition implies, therefore, that every Type B solution always has at least two congruences associated with it, as we already saw in the case of 2521, which possessed a single Type B solution but two associated congruences, which are $2521\equiv -2 \pmod{87}$ and $2521\equiv -29 \pmod{1275}$. It is a result that can be applied outside the context of the conjecture as a proposition in the field of modular arithmetic.

\begin{propo}
Let $p$ be prime, then $p$ possesses a solution that is both of Type A and B if and only if $p\equiv -1 \pmod{4d-1}$ for some $d \in \mathbb{N}$.  
\end{propo}

\begin{proof}
    If $p\equiv -1 \pmod{4d-1}$ then, by (\ref{T4}), it has a Type B solution. Moreover, as $-4d\equiv -1 \pmod{4d-1}$ by (\ref{T5}) that same congruence also generates a Type A solution, which is the one we already had.
\\ \\
Now suppose that p has a solution that is both of Type A and B, i.e., $$(x,y,z)=(du,dv,duv)=(DUV,DUp,DVp)$$ with $d,u,v,D,U,V\in\mathbb{N}$. We can consider without loss of generality that $u<v$. From the Type A structure of the solution we obtain that $xy=dz$, and from the Type B structure we obtain that $yz=Dxp^2$. This implies that $y^2=Ddp^2$. Substituting into the Type B solution we have that $$D^2U^2p^2=Ddp^2$$ and therefore $DU^2=d$, which implies that $U^2=\frac{d}{D}$ and, in particular, that $D\mid d$. Substituting $y^2=Ddp^2$ into the Type A structure we obtain that $$dv^2=Dp^2$$ Since $p\nmid u$  we know that $p\mid v$ and that implies that $\frac{D}{d}=\frac{v^2}{p^2}\in\mathbb{N}$ and $d\mid D$. We conclude that $d=D$ and this automatically implies that $v=p$ and $U=1$. This implies that the solution has the form $(du,dp,dup)$. Since it is of Type B, we have, by (\ref{T4}), that necessarily $p\equiv -1 \pmod{4d-1}$.
\end{proof}

\section{Appendix I: Solutions of Type C?}

As was said in Remark 1 it happens that, if we have the polynomial $R(x,y,t,z)=(4xyz-1)(3+4t)-4x^2y$ that has as its image all the values that possess a Type II solution, the Type A and B solutions are related to it since the Type A solutions are obtained in the particular case in which $x=1$ while the Type B solutions are obtained when $z=1$. This motivates the idea of considering what happens in the case in which $y=1$, for which we obtain the polynomial $R(x,1,t,z)=(4xz-1)(3+4t)-4x^2$. We will study it in the following theorem that encompasses this new type of solutions.

\begin{theor}
\label{T9}
Let $n \in \mathbb{N}$ be of the form $n=4k+1$. They are equivalent:

\begin{enumerate}[label=(\roman*)]
    \item There exists $(x_0,z_0,t_0)$ such that $n=R(x_0,1,t_0,z_0)$;
    \item There exist $d,m\in\mathbb{N}$ such that $-n\equiv 4d^2 \pmod{4dm-1}$;
    \item n has a solution of the Erdos-Straus conjecture with the form $(uv,uwn,vwn)$, with $u,v,w\in\mathbb{N}$;
    \item For a certain $t\geq0$ exists $a,b\in\mathbb{N}$ such that $ab=k+1+t$, $3+4t\mid a+b$.
\end{enumerate}

\end{theor}

We will consider this new solution as a Type C solution.

\begin{proof}
    First we will see that (i) implies (ii). We know that $n$ can be written as $n=(4x_0z_0-1)(3+4t_0)-4x_0^2$, then it is automatically satisfied that $$n\equiv -4x_0^2 \pmod{4x_0z_0-1}$$ and so we have (ii) automatically by identifying $d=x_0$, $m=z_0$.
\\ \\
To show that (ii) implies (iii) we know that $-n\equiv 4d^2 \pmod{4dm-1}$ and therefore $-nm\equiv d \pmod{4dm-1}$. This implies that $d+nm=v(4dm-1)$ for a certain $v\in\mathbb{N}$, so we have that $4dmv=d+nm+v$. Dividing by $dmvn$ we have that $$\frac{4}{n}=\frac{1}{mvn}+\frac{1}{dv}+\frac{1}{dmn}$$ which is of the form requested with $u=d$, $w=m$.
\\ \\
Now we start from (iii) and assume that we have a solution of the form $(uv,uwn,vwn)$, so multiplying by $dmvn$ we have again that $4uwv=u+nw+v$. This implies that $v(4uw-1)=u+nw$ and therefore that $nw\equiv -u \pmod{4uw-1}$. If we substitute $n=4k+1$ this leads us to that $$4wk\equiv -u-v \pmod{4uw-1}$$ and therefore that $$k\equiv -u^2-uw \pmod{4uw-1}$$ Translating the congruence to an equality leads to the fact that $k+u^2+uw=d(4uw-1)$ for a certain $d\in\mathbb{N}$ and hence $$k+d=4duw-uw-u^2$$ which can be rewritten as $$k+d=u(w(4d-1)-u)$$ Since we know that $d$ is a natural number we can make the change of variable $d=t+1$, $t\geq0$ and we have then that the equality can be written as $$k+1+t=u(w(3+4t)-u)$$  This implies that $k+1+t$ can be decomposed as the product of two divisors $a=u$, $b=w(3+4t)-u$, such that $a+b=w(3+4t)$ and therefore $3+4t\mid a+b$, as we were looking for.
\\ \\
To prove finally that (iv) implies (i), we assume that we have two divisors $a,b$ of $k+1+t$ for a certain $t\geq0$ such that $3+4t\mid a+b$. Undoing the same math as above we have that $$k+1+t=u(w(3+4t)-u)$$ and therefore that $$k+1+t=3uw+4uwt-u^2$$ This leads us to the fact that $$4k+4+4t=12uw+16uwt-4u^2$$ and therefore $$4k+1=12uw+16uwt-(3+4t)-4u^2$$ which by rearranging the terms leaves us with the expression $$n=(4uw-1)(3+4t)-4u^2$$ which shows that $n=R(u,1,t,w)$ and therefore that $n$ belongs to the image of the reduced polynomial $n=R(x,1,t,z)$, which was just what was said in section (i).
\end{proof}

\begin{corol}
Given the nature as a quadratic equation of the congruence $-n\equiv 4d^2 \pmod{4dm-1}$ and taking into account the notation of Legendre's symbol and that it always happens that $$\left(\frac{-1}{4u-1}\right)=-1$$ we have then that a necessary condition for a number $n$ of the form $n=4k+1$ to have a solution of Type C will be that there exists a certain $u$ for which $n$ is not a quadratic residue in $\mathbb{Z}_{4u-1}$. This will imply that $-n$ will possess not only one but two solutions to the equation $-n\equiv x^2 \pmod{4u-1}$. If we denote by $x_0$ and $4u-1-x_0$ those solutions we have that exactly one of them will be even and the other odd and, therefore, either $\frac{x_0}{2}$ or $\frac{4u-1-x_0}{2}$  will be a natural number. Then it will be an indispensable condition, for a solution of Type C to exist, that this number be a divisor of $u$.    
\end{corol}

Although this type of solution does not seem  necessary in experimental terms to cover all the prime values congruent to 1 modulo 4, it can be added to the conjecture we already have, by extending the system of congruences without any cost:

\begin{conje}
Let be $p \in \mathbb{N}$ prime, then exists $d,n \in \mathbb{N}$ such that $p\equiv -4d \pmod{4dn-1}$, $p\equiv -n \pmod{4dn-1}$ or $p\equiv -4d^2 \pmod{4dn-1}$. If this result is true, then this congruence system covers all primes and the Erdos-Straus conjecture is true.    
\end{conje}

Another immediate implication, thanks to the fact that (\ref{T4}) is true for all natural numbers, not just prime numbers, is that there exist infinite pairs of values of the form $4k+1$ such that their distance is equal to $4s$, with $s\geq 1$ chosen at our option and satisfying the conjecture. Indeed, if we call $$S(x,y,t)=R(x,1,t,y)=(4xy-1)(3+4t)-4x^2$$ then we have that $S(x,y,t)-P(x,y,t)=4x^2(y-1)$, and given $s\in\mathbb{N}$, if we define $n=S(1,s+1,t)$ and $n'=P(1,s+1,t)$ then it is immediate to check that $n-n'=4s$ for every value $t\geq 0$ and we have made the infinite pairs requested, which satisfy the Erdos-Straus conjecture by being image of the polynomials mentioned above.

\section{Appendix II: Reduction to Egyptian Numbers of Order 2}

Although the values that have the form $n=4k+1$ turn out to be the most problematic to study, we can also take advantage of their particular structure to reach results that only they can fulfill, like this theorem that proves the existence of infinite consecutive values that satisfy the Erdos-Straus Conjecture, something that Manuel Bello Hernández, Manuel Benito and Emilio Fernández already proved in \cite{bello-hernandez_egyptian_2012}, but that is shown here in a much simpler and direct way and by means of an ingenious transformation.

\begin{theor}
\label{T10}
Given $n\in\mathbb{N}$, there exists a chain of $n$ consecutive natural numbers that satisfies the Erdos-Straus conjecture.    
\end{theor}

\begin{proof}
    Since identities are known for even values and of the form $n=4k+3$, it will suffice to prove this for those of the form $n=4k+1$. First we will show that, if there exists $d\in\mathbb{N}$ such that $\frac{4d-1}{k+d}$ is Egyptian of order 2, then $n=4k+1$ satisfies the Erdos-Straus conjecture.
\\ \\
If $\frac{4d-1}{k+d}=\frac{1}{y}+\frac{1}{z}$ then $\frac{4d}{k+d}=\frac{1}{k+d}+\frac{1}{y}+\frac{1}{z}$. We call $k+d=x$, and therefore $d=x-k$ and it follows that $$\frac{4(x-k)}{x}=\frac{1}{x}+\frac{1}{y}+\frac{1}{z}$$ which implies that $$4=\frac{4k+1}{x}+\frac{1}{y}+\frac{1}{z}$$ We call $n=4k+1$ and we have that $$\frac{4}{n}=\frac{1}{x}+\frac{1}{ny}+\frac{1}{nz}$$ which was what we were looking for.
\\ \\
Now, it is well known that $\frac{a}{b}$ is Egyptian of order 2 if and only if there exist $u,v$ divisors of $b$ such that $a$ divides to $u+v$.  Applied to $\frac{4d-1}{k+d}$ we have that there exist $u,v$ divisors of $k+d$ such that $4d-1$ divides to $u+v$. Note that if $k$ is odd it is automatically fulfilled by taking $d=u=1$, $v=2$.
\\ \\
We consider $k$ of the form $n!-f(n)$, and take $d=f(n)$. For this form to be satisfactory there must exist $u,v$ divisors of $n!$ such that $4f(n)-1\mid u+v$. We look for $f(n)$ to be as large as possible, and that is the case if it is also $u+v$, with $u\neq v$. We take therefore $u+v=n+(n-1)$. Equaling we have that $$4f(n)-1=2n-1$$ And therefore these values $u,v$ will always be possible and valid if $f(n)\leq \lfloor \frac{n}{2} \rfloor$, which implies that all values $4k+1$, with $k$ belonging to the interval $[{n!-\lfloor \frac{n}{2}\rfloor,n!-1}]$, satisfy the Erdos-Straus conjecture. Since we can take $n$ as large as we want and the length of the interval increases with $n$, we can find consecutive values to our liking that satisfy this conjecture.
\end{proof}

\begin{corol}
Proving that there exists $d\in\mathbb{N}$ such that $\frac{4d-1}{k+d}$ is Egyptian of order 2 is equivalent to finding a Type II solution for $4k+1$.
\end{corol}

\begin{proof}
If $\frac{4d-1}{k+d}=\frac{1}{y}+\frac{1}{z}$, then there exists $u,v\mid k+d$, $4d-1\mid u+v$, and therefore $k+d=uvw$, $u+v=t(4d-1)$ for some certain naturals $w,t$. By elementary calculations and calling $n=4k+1$ we obtain that $$n=4uvw-\frac{u+v}{t}$$ Renaming $u=ct-v$, we have that $n=c(4vwt-1)-4v^2w$. By changing the names of the variables, $n=(4xyz-1)c-4x^2y$. Since $n$ is congruent with 1 modulo 4, then $c$ is congruent with 3 modulo 4, and so $c=3+4s$, $s\geq 0$, and $n=R(x,y,z,s)$. By (\ref{T8}), $n$ is then a value with a Type II solution.
\end{proof}

\section{Conclusion}

We have defined two types of solutions, Type A and B, which have allowed us to propose a system of congruences that we conjecture gives a solution of the Erdos-Straus conjecture to all prime numbers. This is one of the most reasonable techniques for solving the Erdos-Straus conjecture: to obtain a system of congruences such that they all have an associated polynomial identity and that they completely cover all prime numbers. The first part, that every congruence has an associated polynomial identity and therefore none of them contain quadratic residues, has already been shown. It remains to be seen whether counterexamples to this new conjecture are found, which is the same as finding some prime number that lacks solutions of both Type A and Type B and, if not found and therefore suspecting that the conjecture is true, whether or not we have created an alternative formulation of the same problem as difficult to prove or more difficult to prove than the original problem. So far we have experimentally proved that our proposed system of congruences gives a solution to the first ten thousand prime numbers without exceptions, which gives us hope for its usefulness as a tool to solve the Erdos-Straus conjecture.
\\ \\
\bibliographystyle{unsrt}

\end{document}